\documentclass[11pt]{amsart}
\usepackage{amssymb}
\usepackage{amsthm}
\usepackage{amsmath}
\usepackage{amsfonts}
\usepackage[mathscr]{eucal}
\usepackage{enumerate}
\usepackage[latin2]{inputenc}

\newtheorem{thm}{Theorem}[section]

\newtheorem{lem}[thm]{Lemma}

\theoremstyle{definition}
\newtheorem{defn}[thm]{Definition}
\theoremstyle{remark}

\numberwithin{equation}{section}

\begin{document}
\title[A characterization of the relative entropies]
{A characterization of the relative entropies}

\author[E.~Gselmann]{Eszter Gselmann}
\address{
Department of Analysis\\
Institute of Mathematics\\
University of Debrecen\\
P. O. Box: 12.\\
Debrecen\\
H--4010\\
Hungary}

\email{gselmann@math.klte.hu}

\author[Gy.~Maksa]{Gyula Maksa}
\address{
Department of Analysis\\
Institute of Mathematics\\
University of Debrecen\\
P. O. Box: 12.\\
Debrecen\\
H--4010\\
Hungary}

\email{maksa@math.klte.hu}
\begin{flushright}
{\tt Manuscript\\
\today}
\end{flushright}

\begin{abstract}
In this note we give a characterization of a family of relative entropies on open domain
depending on a real parameter $\alpha$ based on recursivity and symmetry. In the cases $\alpha=1$ and
$\alpha=0$ we use additionally a weak regularity assumption while in the other cases
no regularity assumptions are made at all.
\end{abstract}

\thanks{This research has been supported by the Hungarian Scientific Research Fund
(OTKA) grants NK 68040 and K 62316. }

\dedicatory{Dedicated to Professor Antal J\'arai on his sixtieth birthday}
\maketitle

\section{Introduction and preliminaries}

Throughout this paper $\mathbb{N}$, $\mathbb{R}$, and $\mathbb{R}_{+}$ will denote the sets of all
positive integers, real numbers, and positive real numbers, respectively. For all $2\leq n\in\mathbb{N}$ let
\[
\Gamma_{n}^{\circ}=\left\{(p_{1}, \ldots, p_{n})\in \mathbb{R}^{n} \vert p_{i}\in\mathbb{R}_{+}, i=1, \ldots, n, \sum_{i=1}^{n}p_{i}=1 \right\}
\]
and
\[
\Gamma_{n}=\left\{(p_{1}, \ldots, p_{n})\in \mathbb{R}^{n} \vert p_{i}\geq 0, i=1, \ldots, n, \sum_{i=1}^{n}p_{i}=1 \right\}.
\]
Furthermore, for a fixed $\alpha\in\mathbb{R},$ define the function
$D_{n}^{\alpha}(\cdot\vert\cdot):\Gamma_{n}^{\circ}\times\Gamma_{n}^{\circ}\rightarrow\mathbb{R}$ by
\begin{equation}
D_{n}^{\alpha}(p_{1}, \ldots, p_{n}\vert q_{1}, \ldots, q_{n})=
-\sum_{i=1}^{n}p_{i}\ln_{\alpha}\left(\frac{q_{i}}{p_{i}}\right)
\end{equation}
where
\[
\mathrm{ln}_{\alpha}(x)=\left\{
\begin{array}{lcl}
\frac{x^{1-\alpha}-1}{1-\alpha},& \text{if}& \alpha\neq 1 \\
\ln(x), & \text{if}& \alpha=1.
\end{array}
\right.
\]
The sequence $(D_{n}^{\alpha})$ is called the Shannon relative entropy (or Kullback-Leibler entropy or Kullback's
directed divergence) if $\alpha=1$, and the Tsallis relative entropy if $\alpha\neq 1$, respectively. $(D_{n}^{1})$
is introduced and extensively discussed in Kullback \cite{Kul59} and Acz\'el--Dar\'oczy \cite{AD75}, respectively.
For $0\leq\alpha\neq 1$, $(D_{n}^{\alpha})$ was introduced and discussed in Shiino \cite{Shi98}, Tsallis \cite{Tsa98},
and Rajagopal--Abe \cite{RA99} from physical point of view, and in Furuichi--Yanagi--Kuriyama \cite{FYK04} and
Furuichi \cite{Fur05} from mathematical point of view, respectively. In \cite{Fur05} and also in Hobson \cite{Hob69},
several fundamental properties of $(D_{n}^{\alpha})$ are listed and it is proved that some of them together determine
$(D_{n}^{\alpha})$, up to a constant factor.

In this note, we follow the method of the basic references \cite{AD75} and Ebanks--Sahoo--Sander \cite{ESS98} of
investigating characterization problems of information measures. We prove a characterization theorem similar to
that of \cite{Hob69} and \cite{Fur05}, and we point out that the regularity conditions (say, continuity) can be
avoided if $\alpha\notin \{0,1\}$, and can essentially be weakened if $\alpha\in \{0,1\}$.

In what follows, a sequence $(I_{n})$ of real-valued functions $I_{n}, (n\geq 2)$ on $\Gamma_{n}^{\circ}\times\Gamma_{n}^{\circ}$
or on $\Gamma_{n}\times\Gamma_{n}$ is called a \emph{relative information measure} on the open or closed domain, respectively.
Our characterization theorem for the Shannon and the Tsallis relative entropies will be based on
the following two properties.

\begin{defn}
Let $\alpha\in\mathbb{R}$. The relative information measure $(I_{n})$ is \emph{$\alpha$--recursive} on the open or closed domain,
if for any $n\geq 3$ and $(p_{1}, \ldots, p_{n}), (q_{1}, \ldots, q_{n})\in\Gamma_{n}^{\circ}$ or $\Gamma_{n}$, respectively,
the identity
\begin{multline*}
I_{n}\left(p_{1}, \ldots, p_{n}\vert q_{1}, \ldots, q_{n}\right)\\
=I_{n-1}\left(p_{1}+p_{2}, p_{3}, \ldots, p_{n}\vert q_{1}+q_{2}, q_{3}, \ldots, q_{n}\right)\\
+(p_{1}+p_{2})^{\alpha}(q_{1}+q_{2})^{1-\alpha}
I_{2}\left(\frac{p_{1}}{p_{1}+p_{2}}, \frac{p_{2}}{p_{1}+p_{2}}\vert
\frac{q_{1}}{q_{1}+q_{2}}, \frac{q_{2}}{q_{1}+q_{2}}\right)
\end{multline*}
holds. We say that $(I_{n})$ is \emph{3-semisymmetric} on the open or closed domain, if
\[
I_{3}\left(p_{1}, p_{2}, p_{3}\vert q_{1}, q_{2}, q_{3}\right)=
I_{3}\left(p_{1}, p_{3}, p_{2}\vert q_{1}, q_{3}, q_{2}\right)
\]
is fulfilled for all $(p_{1}, p_{2}, p_{3}), (q_{1}, q_{2}, q_{3})\in\Gamma_{3}^{\circ}\,\, \text{or}\,\, \Gamma_{3}$,
respectively.
\end{defn}
The following lemma shows how the initial element of an $\alpha$--recursive relative information measure $(I_{n})$ determines
$(I_{n})$ itself.
\begin{lem}\label{L1.1}
Let $\alpha\in\mathbb{R}$ and assume that the relative information measure
$(I_{n})$ is $\alpha$--recursive on the open domain and define the function $f: ]0,1[^{2}\to\mathbb{R}$ by
\[
f(x,y)=I_{2}(1-x,x \vert 1-y,y). \qquad (x,y\in ]0, 1[)
\]
Then, for all $n\geq 3$ and for arbitrary,
$(p_{1}, \ldots, p_{n}), (q_{1}, \ldots, q_{n})\in\Gamma_{n}^{\circ}$
\begin{multline*}
I_{n}(p_{1}, \ldots, p_{n} \vert q_{1}, \ldots, q_{n})\\
=\sum_{i=2}^{n}(p_{1}+p_{2}+\ldots+p_{i})^{\alpha}(q_{1}+q_{2}+\ldots+q_{i})^{1-\alpha}
f\left(\frac{p_{i}}{p_{1}+p_{2}+\ldots+p_{i}},\frac{q_{i}}{q_{1}+q_{2}+\ldots+q_{i}}\right)
\end{multline*}
holds.
\end{lem}

\begin{proof}
The proof runs by induction on $n$. If we use the $\alpha$--recursivity of
$(I_{n})$ and the definition of the function $f$, we obtain that
\begin{multline*}
I_{3}(p_{1}, p_{2}, p_{3}\vert q_{1}, q_{2}, q_{3})\\
=I_{2}(p_{1}+p_{2}, p_{3}\vert q_{1}+q_{2}, q_{3})+
(p_{1}+p_{2})^{\alpha}(q_{1}+q_{2})^{1-\alpha}
I_{2}\left(\frac{p_{1}}{p_{1}+p_{2}}, \frac{p_{2}}{p_{1}+p_{2}}\bigg\vert\frac{q_{1}}{q_{1}+q_{2}}, \frac{q_{2}}{q_{1}+q_{2}}\right)\\
=\sum_{i=2}^{3}(p_{1}+\ldots+p_{i})^{\alpha}(q_{1}+\ldots+q_{i})^{1-\alpha}
f\left(\frac{p_{i}}{p_{1}+\ldots p_{i}},\frac{q_{i}}{q_{1}+\ldots+q_{i}}\right)
\end{multline*}
is fulfilled for all $(p_{1}, p_{2}, p_{3}), (q_{1}, q_{2}, q_{3})\in\Gamma^{\circ}_{3}$, that is,
the statement is true for $n=3$. Assume now that the statement holds for some $3<n\in\mathbb{N}$. We will prove that
in this case the proposition holds also for $n+1$. Let $(p_{1}, \ldots, p_{n+1}), (q_{1}, \ldots, q_{n+1})\in\Gamma_{n}^{\circ}$ be
arbitrary. Then, the $\alpha$--recursivity and the induction hypothesis together imply that
\begin{multline*}
I_{n+1}(p_{1}, \ldots, p_{n+1} \vert q_{1}, \ldots, q_{n+1})\\
=I_{n}(p_{1}+p_{2}, \ldots, p_{n+1}\vert q_{1}+q_{2}, \ldots, q_{n+1})+
(p_{1}+p_{2})^{\alpha}(q_{1}+q_{2})^{1-\alpha}
I_{2}\left(\frac{p_{1}}{p_{1}+p_{2}}, \frac{p_{2}}{p_{1}+p_{2}}\bigg\vert\frac{q_{1}}{q_{1}+q_{2}}, \frac{q_{2}}{q_{1}+q_{2}}\right)\\
=\sum_{n=3}^{n+1}((p_{1}+p_{2})+p_{3}\ldots+p_{i})^{\alpha}((q_{1}+q_{2})+p_{3}+\ldots+q_{i})^{1-\alpha}
f\left(\frac{p_{i}}{(p_{1}+p_{2})+\ldots+p_{i}},\frac{q_{i}}{(q_{1}+q_{2})+\ldots+q_{i}}\right)\\
+(p_{1}+p_{2})^{\alpha}(q_{1}+q_{2})^{1-\alpha}
I_{2}\left(\frac{p_{1}}{p_{1}+p_{2}}, \frac{p_{2}}{p_{1}+p_{2}}\bigg\vert\frac{q_{1}}{q_{1}+q_{2}}, \frac{q_{2}}{q_{1}+q_{2}}\right)\\
=\sum_{i=2}^{n+1}(p_{1}+p_{2}+\ldots+p_{i})^{\alpha}(q_{1}+q_{2}+\ldots+q_{i})^{1-\alpha}
f\left(\frac{p_{i}}{p_{1}+p_{2}+\ldots+p_{i}},\frac{q_{i}}{q_{1}+q_{2}+\ldots+q_{i}}\right),
\end{multline*}
that is, the statement holds for $n+1$ instead of $n$, which ends the proof.
\end{proof}
\section{The characterization}
We begin with the following
\begin{thm}
For any $\alpha\in\mathbb{R}$ the relative entropy $(D_{n}^{\alpha})$ is an
$\alpha$--recursive relative information measure.
\end{thm}

\begin{proof}
In the proof, we will use several times the identities
\setlength\arraycolsep{2pt}
\begin{eqnarray*}
\ln_{\alpha}(xy)&=&\ln_{\alpha}(x)+\ln_{\alpha}(y)+(1-\alpha)\ln_{\alpha}(x)\ln_{\alpha}(y)\\
\ln_{\alpha}\left(\frac{1}{x}\right)&=&-x^{\alpha-1}\ln_{\alpha}(x).
\end{eqnarray*}
which hold for all $\alpha\in\mathbb{R}$ and $x,y\in \mathbb{R}_{+}$. Let $n\geq 3$ and
$(p_{1}, \ldots, p_{n}), (q_{1}, \ldots, q_{n})\in\Gamma_{n}^{\circ}$ be arbitrary. Then
\begin{multline*}
(p_{1}+p_{2})^{\alpha}(q_{1}+q_{2})^{1-\alpha}D_{2}\left(\frac{p_{1}}{p_{1}+p_{2}}, \frac{p_{2}}{p_{1}+p_{2}}\bigg\vert\frac{q_{1}}{q_{1}+q_{2}}, \frac{q_{2}}{q_{1}+q_{2}}\right) \\
=(p_{1}+p_{2})^{\alpha}(q_{1}+q_{2})^{1-\alpha}
\left(-\frac{p_{1}}{p_{1}+p_{2}}\ln_{\alpha}\left(\frac{p_{1}+p_{2}}{q_{1}+p_{2}}\frac{q_{1}}{p_{1}}\right)
-\frac{p_{2}}{p_{1}+p_{2}}\ln_{\alpha}\left(\frac{p_{1}+p_{2}}{q_{1}+q_{2}}\frac{q_{2}}{p_{2}}\right)\right)\\
=(p_{1}+p_{2})^{\alpha}(q_{1}+q_{2})^{1-\alpha}
\left(-\ln_{\alpha}\left(\frac{p_{1}+p_{2}}{q_{1}+q_{2}}\right)
+\left(1+(1-\alpha)\ln_{\alpha}\left(\frac{p_{1}+p_{2}}{q_{1}+q_{2}}\right)\right)\right.
\times \\
\times \left.
\left(-\frac{p_{1}}{p_{1}+p_{2}}\ln_{\alpha}\left(\frac{q_{1}}{p_{1}}\right)
-\frac{p_{2}}{p_{1}+p_{2}}\ln_{\alpha}\left(\frac{q_{2}}{p_{2}}\right)\right)\right)
\\
=
(p_{1}+p_{2})\ln_{\alpha}\left(\frac{q_{1}+q_{2}}{p_{1}+p_{2}}\right)
+
\left[\left(\frac{q_{1}+q_{2}}{p_{1}+p_{2}}\right)^{1-\alpha}-(1-\alpha)\ln_{\alpha}\left(\frac{q_{1}+q_{2}}{p_{1}+p_{2}}\right)\right]
\left[-p_{1}\ln_{\alpha}\frac{q_{1}}{p_{1}}-p_{2}\ln_{\alpha}\frac{q_{2}}{p_{2}}\right]
\\
=
(p_{1}+p_{2})\ln_{\alpha}\left(\frac{q_{1}+q_{2}}{p_{1}+p_{2}}\right)-p_{1}\ln_{\alpha}\left(\frac{q_{1}}{p_{1}}\right)-
p_{2}\ln_{\alpha}\left(\frac{q_{2}}{p_{2}}\right)
\\
=
D_{n}(p_{1}, \ldots, p_{n}\vert q_{1}, \ldots, q_{n})
-
D_{n-1}(p_{1}+p_{2}, \ldots, p_{n}\vert q_{1}+q_{2}+\ldots, q_{n}).
\end{multline*}
Therefore the relative entropy $(D_{n}^{\alpha})$ is $\alpha$--recursive, indeed.
\end{proof}

Obviously $(D_{n}^{\alpha})$ is 3-semisymmetric, and for arbitrary $\gamma\in\mathbb{R},\,\,$
$(\gamma D_{n}^{\alpha})$ is $\alpha$--recursive and
3-semisymmetric, as well. Before dealing with the converse we need two lemmas about \emph{logarithmic} functions.
A function $\ell:\mathbb{R}_{+}\to\mathbb{R}$ is logarithmic if $\ell(xy)=\ell(x)+\ell(y)$ for all $x,y\in\mathbb{R}_{+}.$ If a logarithmic function $\ell$
is bounded above or below on a set of positive Lebesgue measure then $\ell(x)=c\ln(x)$ for all $x\in\mathbb{R}_{+}$ with some $c\in\mathbb{R}$
(see \cite{Kuc85}, Theorem 5 and Theorem 8 on pages 311, 312). The concept of \emph{real derivation} will also be needed. The
function $d:\mathbb{R}\to\mathbb{R}$ is a real derivation if it is both \emph{additive}, i.e. $d(x+y)=d(x)+d(y)$ for all $x,y\in\mathbb{R}$, and satisfies
the functional equation $d(xy)=xd(y)+yd(x)$ for all $x,y\in\mathbb{R}$. It is somewhat surprising that there are non-identically zero
real derivations (see \cite{Kuc85}, Theorem 2 on page 352). If $d$ is a real derivation then the function
$x\mapsto \frac{d(x)}{x}, x\in\mathbb{R}_{+}$ is logarithmic. Therefore it is easy to see that the real derivation is identically zero
if it is bounded above or below on a set of positive Lebesgue measure.

\begin{lem}\label{L2.2}
Suppose that the logarithmic function $\ell:\mathbb{R}_{+}\to \mathbb{R}$ satisfies the equality
\begin{equation}\label{eq:2.1}
x\ell(x)+(1-x)\ell(1-x)=0. \qquad (x\in ]0,1[)
\end{equation}
Then there exists a real derivation $d:\mathbb{R}\to\mathbb{R}$ such that
\begin{equation}\label{eq:2.2}
x\ell(x)=d(x). \qquad (x\in \mathbb{R}_{+})
\end{equation}
\end{lem}

\begin{proof}
Let $x,y\in\mathbb{R}_{+}.$ Then, by \eqref{eq:2.1} and by using the properties of the logarithmic function, we have that
\begin{multline*}
0=\frac{x}{x+y}\ell\left(\frac{x}{x+y}\right)+\frac{y}{x+y}\ell\left(\frac{y}{x+y}\right)\\
=\frac{x}{x+y}\left(\ell(x)-\ell(x+y)\right)+\frac{y}{x+y}\left(\ell(y)-\ell(x+y)\right)\\
=\frac{1}{x+y}\left(x\ell(x)+y\ell(y)-(x+y)\ell(x+y)\right).
\end{multline*}
This shows that the function $x\mapsto x\ell(x), x\in\mathbb{R}_{+}$ is additive on $\mathbb{R}_{+}$. Hence, by the well-known
extension theorem (see e.g. \cite{Kuc85}, Theorem 1 on page 471), there exists an additive function
$d:\mathbb{R}\to\mathbb{R}$ such that \eqref{eq:2.2} holds. Since $\ell$ is logarithmic, this implies that $d(xy)=xd(y)+yd(x)$
holds for all $x,y\in\mathbb{R}_{+}$. On the other hand, $d$ is odd thus this equation holds also for all $x,y\in\mathbb{R}$,
that is, $d$ is a real derivation.
\end{proof}

\begin{lem}\label{L2.3}
Suppose that $\ell:\mathbb{R}_{+}\to \mathbb{R}$ is a logarithmic function and the function
$g_{0}$ defined on the interval $]0,1[$ by
\begin{equation*}
g_{0}(x)=x\ell(x)+(1-x)\ell(1-x)
\end{equation*}
is bounded on a set of positive Lebesque measure. Then there exist a real number $\beta$
and a real derivation $d:\mathbb{R}\to\mathbb{R}$ such that
\begin{equation}\label{eq:2.3}
x\ell(x)+\beta x\ln(x)=d(x). \qquad (x\in \mathbb{R}_{+})
\end{equation}
\end{lem}

\begin{proof}
Define the function $g$ on the interval $[0,1]$ by $g(0)=g(1)=0$ and, for $x\in ]0,1[$, by
\begin{equation*}
g(x)=-\frac{g_{0}(x)}{\ell(2)}\,\,\text{if}\,\,\ell(2)\neq 0 \,\,\text{and}\,\,
g(x)=g_{0}(x)-x\log_{2}(x)-(1-x)\log_{2}(1-x) \,\,\text{if}\,\,\ell(2)=0.
\end{equation*}
Then $g$ is a symmetric information function (see \cite{AD75}, (3.5.33) Theorem on page 100) which,
by our assumption, is bounded on a set of positive Lebesque measure. Therefore, applying a theorem of
Diderrich \cite{Did86}, we obtain that
\begin{equation*}
g(x)=-x\log_{2}(x)-(1-x)\log_{2}(1-x). \qquad (x\in ]0,1[)
\end{equation*}
For a short proof of Diderrich's theorem see also \cite{Mak80} in which an idea of
J\'arai \cite{Jar79} proved to be very efficient. Taking into consideration the definition of $g$
and applying Lemma \ref{L2.2}, we get \eqref{eq:2.3} with suitable $\beta \in\mathbb{R}$.
\end{proof}

Now we are ready to prove our main result.

\begin{thm}\label{T2.4}
Let $\alpha\in\mathbb{R}, \,\, (I_{n})$ be an $\alpha$-recursive and $3$-semisymmetric relative information measure
on the open domain, and $f(x, y)=I_{2}(1-x, x|1-y, y),\,\, x,y\in ]0,1[$. Furthermore, suppose that
\begin{equation}\label{eq:2.4}
I_{2}(p_{1}, p_{2}|p_{1}, p_{2})=0.  \qquad ((p_{1},p_{2})\in \Gamma_{2})
\end{equation}
If $\alpha\notin \{0,1\}$ then $(I_{n})=(\gamma D_{n}^{\alpha})$ for some $\gamma \in\mathbb{R}$.
\\
If $\alpha=1$ and there exists a point $(u,v)\in ]0,1[^{2}$
such that the function $f(\cdot,v)$ is bounded on a set of positive Lebesgue measure and the function $f(u,\cdot)$
is bounded above or below on a set of positive Lebesgue measure then $(I_{n})=(\gamma D_{n}^{1})$ for some $\gamma \in\mathbb{R}$.
\\
And finally, if
$\alpha=0$ and there exists a point $(u,v)\in ]0,1[^{2}$ such that the function $f(\cdot,v)$ is bounded above or below
on a set of positive Lebesgue measure and the function $f(u,\cdot)$ is bounded on a set of positive Lebesgue measure then $(I_{n})=(\gamma D_{n}^{0})$
for some $\gamma \in\mathbb{R}$.
\end{thm}

\begin{proof}
Applying Theorem 4.2.3. on page 87 of \cite{ESS98} with $M(x,y)=x^{\alpha}y^{1-\alpha},\,\,x,y\in\mathbb{R}_{+}$ and taking into
consideration Lemma 1.2.12. on page 16 of \cite{ESS98}, (see also \cite{A81}), we have that
\begin{equation}\label{eq:2.5}
I_{n}\left(p_{1}, \ldots, p_{n}\vert q_{1}, \ldots, q_{n}\right)=bp_{1}^{\alpha}q_{1}^{1-\alpha}+c\sum_{i=2}^{n}p_{i}^{\alpha}q_{i}^{1-\alpha}-b
\end{equation}
in case $\alpha\notin \{0,1\}$,
\begin{equation}\label{eq:2.6}
I_{n}\left(p_{1}, \ldots, p_{n}\vert q_{1}, \ldots, q_{n}\right)=\sum_{i=1}^{n}p_{i}(\ell_{1}(p_{i})+\ell_{2}(q_{i}))+c(1-p_{1})
\end{equation}
in case $\alpha=1$, and
\begin{equation}\label{eq:2.7}
I_{n}\left(p_{1}, \ldots, p_{n}\vert q_{1}, \ldots, q_{n}\right)=\sum_{i=1}^{n}q_{i}(\ell_{1}(p_{i})+\ell_{2}(q_{i}))+c(1-q_{1})
\end{equation}
in case $\alpha=0$ for all $n\geq 2, (p_{1}, \ldots, p_{n}), (q_{1}, \ldots, q_{n})\in\Gamma_{n}^{\circ}$ with some $b,c\in\mathbb{R}$ and
logarithmic functions $\ell_{1}, \ell_{2}:\mathbb{R}_{+}\to\mathbb{R}$.

Now we utilize our further conditions on $(I_{n})$. In case $\alpha\notin \{0,1\}$, \eqref{eq:2.5} with $n=2$ and \eqref{eq:2.4} imply that
$0=bp_{1}+cp_{2}-b$ for all $(p_{1}, p_{2})\in \Gamma_{2}$ whence $b=c$ follows. Thus, by \eqref{eq:2.5}, we obtain that
$(I_{n})=(\gamma D_{n}^{\alpha})$ with $\gamma=(\alpha-1)^{-1}$. In case $\alpha=1$, \eqref{eq:2.6} with $n=2$ and \eqref{eq:2.4} imply that
\begin{equation*}
0=p_{1}\ell(p_{1})+p_{2}\ell(p_{2})+c(1-p_{1}). \qquad ((p_{1}, p_{2})\in \Gamma_{2})
\end{equation*}
where $\ell=\ell_{1}+\ell_{2}$. Therefore $c=0$, and, by Lemma \ref{L2.2}, we get that $x\ell_{2}(x)=-x\ell_{1}(x)+d_{1}(x)$ for all
$x\in\mathbb{R}_{+}$ and for some real derivation $d_{1}:\mathbb{R}\to\mathbb{R}$. Thus
\begin{equation*}
f(x,y)=x\ell_{1}\left(\frac{x}{y}\right)
+(1-x)\ell_{1}\left(\frac{1-x}{1-y}\right)
+\left(\frac{x}{y}-\frac{1-x}{1-y}\right)d_{1}(y).
\qquad (x,y\in ]0,1[)
\end{equation*}
Since the function $f(\cdot,v)$ is bounded on a set of positive Lebesque measure, we get that the function
$x\mapsto x\ell_{1}(x)+(1-x)\ell_{1}(1-x), \, x\in]0,1[$ has the same property. Thus, by Lemma \ref{L2.3},
\[
x\ell_{1}(x)+\beta x\ln(x)=d_{2}(x), \qquad (x\in \mathbb{R}_{+})
\]
for some $\beta \in\mathbb{R}$ and derivation $d_{2}:\mathbb{R}\to\mathbb{R}$.
Hence
\begin{equation*}
f(x,y)=-\beta x\ln\left(\frac{x}{y}\right)
-\beta(1-x)\ln\left(\frac{1-x}{1-y}\right)
-\left(\frac{x}{y}-\frac{1-x}{1-y}\right)(d_{2}(y)-d_{1}(y)).
\qquad (x,y\in ]0,1[)
\end{equation*}
However, $f(u,\cdot)$ is bounded above or below on a set of positive Lebesgue measure for some $u\in ]0,1[$ thus the
derivation $d_{2}-d_{1}$ has the same property, so $d_{2}-d_{1}=0$. Therefore
\begin{equation*}
f(x,y)=-\beta x\ln\left(\frac{x}{y}\right)
-\beta(1-x)\ln\left(\frac{1-x}{1-y}\right)
\qquad (x,y\in ]0,1[)
\end{equation*}
and the statement follows from Lemma \ref{L1.1} with a suitable $\gamma \in\mathbb{R}$.
The case $\alpha=0$ can be handled similarly by interchanging the role of the distributions $(p_{1}, \ldots, p_{n})$ and
$(q_{1}, \ldots, q_{n})$ and of the logarithmic functions $\ell_{1}$ and $\ell_{2},$ respectively.
\end{proof}

\section{Connections to known characterizations}

In this section we will point out some connections between our
characterization theorem and other statements. Here we deal especially with the results
of Hobson \cite{Hob69} and Furuichi \cite{Fur05}. They considered the relative information measure on the closed domain.
In this case, however the expressions $\frac{0}{0+0}, \frac{0}{0+\ldots +0}, 0^{\alpha}, 0^{1-\alpha}$ can appear. Therefore,
in the remaining part of the paper the conventions $\frac{0}{0+0}=\frac{0}{0+\ldots +0}=0^{\alpha}=0^{1-\alpha}=0$ are always adapted
(see also \cite{AK82}).

We begin with several definitions.

\begin{defn}
The relative information measure $(I_{n})$ on the closed domain
is said to be \emph{expansible}, if
\[
I_{n+1}\left(p_{1}, \ldots, p_{n}, 0\vert q_{1}, \ldots, q_{n}, 0\right)=
I_{n}\left(p_{1}, \ldots, p_{n} \vert q_{1}, \ldots, q_{n} \right)
\]
is satisfied for all $n\geq 2$ and
$(p_{1}, \ldots, p_{n}), (q_{1}, \ldots, q_{n})\in\Gamma_{n}$.
\\
The relative information measure is called \emph{decisive}, if
\[
I_{2}(1, 0\vert 1, 0)=0
\]
holds.
\\
Let $\alpha\in\mathbb{R}$ be arbitrarily fixed,
we say that the relative information measure $(I_{n})$
satisfies the \emph{generalized additivity}, if
for all $n, m\geq 2$ and for arbitrary
$(p_{1, 1}, \ldots, p_{1, m}, \ldots, \ldots, p_{n, 1}, \ldots, p_{n, m}), \allowbreak
(q_{1, 1}, \ldots, q_{1, m}, \ldots, \ldots, q_{n, 1}, \ldots, q_{n, m}) \in\Gamma_{nm}$
(or $\Gamma_{nm}^{\circ}$)
\begin{multline*}
I_{nm}\left(p_{1, 1}, \ldots, p_{1, m}, \ldots, \ldots, p_{n, 1}, \ldots, p_{n, m}\vert
q_{1, 1}, \ldots, q_{1, m}, \ldots, \ldots, q_{n, 1}, \ldots, q_{n, m}\right)
\\=
I_{n}(P_{1}, \ldots, P_{n}\vert Q_{1}, \ldots Q_{n})+
\sum_{i=1}^{n}P_{i}^{\alpha}Q_{i}^{1-\alpha}
I_{m}\left(\frac{p_{i, 1}}{P_{i}}, \ldots, \frac{p_{i, m}}{P_{i}}
\vert
\frac{q_{i, 1}}{Q_{i}}, \ldots, \frac{q_{i, m}}{Q_{i}}
\right)
\end{multline*}
is fulfilled, where $P_{i}=\sum_{j=1}^{m}p_{i, j}$ and
$Q_{i}=\sum_{j=1}^{m}q_{i, j}$, $i=1, \ldots, n$.
\end{defn}

A lengthy but simple calculation shows that the relative information measure $(D_{n}^{\alpha})$
fulfills all of the above listed criteria. As well as Hobson \cite{Hob69} and Furuichi \cite{Fur05},
we would like to investigate the converse direction.
More precisely, the question is whether the generalized additivity property determines
$(D_{n}^{\alpha})$ up to a multiplicative constant. In general this is not true.
Since let us observe that in case we consider the generalized additivity on the
open domain $\Gamma_{n}^{\circ}$ then this property is insignificant for
$I_{n}$ if $n$ is a prime.
Nevertheless, on the closed domain this property is well--treatable.
In this case we can prove the following.

\begin{lem}\label{L2.4}
If the relative information measure $(I_{n})$ on the closed domain
is expansible and satisfies the general additivity property with a certain
$\alpha\in\mathbb{R}$, then it is also decisive and $\alpha$--recursive.
\end{lem}

\begin{proof}
Firstly, we will show, that the generalized additivity and the expansibility
implies that the relative information measure $(I_{n})$
is decisive.
Indeed, if we use the generalized additivity with the choice $n=m=2$ and
$(p_{1}, p_{2}, p_{3}, p_{4})=(q_{1}, q_{2}, q_{3}, q_{4})=(1, 0, 0, 0)$, then we
get that
\[
I_{4}(1, 0, 0, 0 \vert 1, 0, 0, 0)=
I_{2}(1, 0\vert 1, 0)+I_{2}(1, 0)
\]
holds. On the other hand, $(I_{n})$ is expansible, therefore
$I_{4}(1, 0, 0, 0\vert 1, 0, 0, 0)=I_{2}(1, 0\vert 1, 0)$.
Thus $I_{2}(1, 0\vert 1, 0)=0$ follows, so $(I_{n})$
is decisive.

Now we will prove the $\alpha$--recursivity of $(I_{n})$.
Let $(r_{1}, \ldots, r_{n}), (s_{1}, \ldots, s_{n})\in\Gamma_{n}$ and
use the generalized additivity with the following substitution
\[
p_{1, 1}=r_{1},\quad  p_{1, 2}=r_{2},\quad  p_{i, 1}=r_{i+1}, i=2, \ldots, n-1, \quad
p_{i, j}=0 \quad \text{otherwise}
\]
and
\[
q_{1, 1}=s_{1},\quad  q_{1, 2}=s_{2},\quad  q_{i, 1}=s_{i+1}, i=2, \ldots, n-1,\quad
q_{i, j}=0 \quad \text{otherwise}
\]
to derive
\begin{multline*}
I_{nm}(r_{1}, r_{2}, 0, \ldots, 0, r_{3}, 0, \ldots, 0, r_{n}, 0, \ldots, 0
\vert s_{1}, s_{2}, 0, \ldots, 0, s_{3}, 0, \ldots, 0, s_{n}, 0, \ldots, 0)
\\
=
I_{n}(r_{1}+r_{2}, r_{3}, \ldots, r_{n}, 0 \vert s_{1}+s_{2}, s_{3}, \ldots, s_{n}, 0)
\\
+
(r_{1}+r_{2})^{\alpha}(s_{1}+s_{2})^{1-\alpha}I_{2}\left(\frac{r_{1}}{r_{1}+r_{2}}, \frac{r_{2}}{r_{1}+r_{2}}\bigg \vert
\frac{s_{1}}{s_{1}+s_{2}}, \frac{s_{2}}{s_{1}+s_{2}}\right)
\\
+\sum_{j=3}^{n}r_{j}^{\alpha}q_{j}^{1-\alpha}I_{m}(1, 0, \ldots, 0\vert 1, 0, \ldots, 0).
\end{multline*}
After using that $(I_{n})$ is expansible and decisive, we
obtain the $\alpha$--recursivity.
\end{proof}

In view of Theorem \ref{T2.4}. and Lemma \ref{L2.4}. the following characterization
theorem follows easily.

\begin{thm}\label{T2.5}
Let $\alpha\in\mathbb{R}, \,\, (I_{n})$ be an expansible and $3$-semisymmetric relative information measure
which also satisfies the generalized additivity property on $\Gamma_{n}$ with the parameter $\alpha$ and
let $f(x,y)=I_{2}(1-x,x|1-y,y),\,\, x,y\in ]0,1[$. Additionally, suppose that
\begin{equation}
I_{2}(p_{1}, p_{2}|p_{1}, p_{2})=0.  \qquad ((p_{1},p_{2})\in \Gamma_{2})
\end{equation}
If $\alpha\notin \{0,1\}$ then $(I_{n})=(\gamma D_{n}^{\alpha})$ for some $\gamma \in\mathbb{R}$.
\\
If $\alpha=1$ and there exists a point $(u,v)\in ]0,1[^{2}$
such that the function $f(\cdot,v)$ is bounded on a set of positive Lebesgue measure and the function $f(u,\cdot)$
is bounded above or below on a set of positive Lebesgue measure then $(I_{n})=(\gamma D_{n}^{1})$ for some $\gamma \in\mathbb{R}$.
\\
And finally, if
$\alpha=0$ and there exists a point $(u,v)\in ]0,1[^{2}$ such that the function $f(\cdot,v)$ is bounded above or below
on a set of positive Lebesgue measure and the function $f(u,\cdot)$ is bounded on a set of positive Lebesgue measure then $(I_{n})=(\gamma D_{n}^{0})$
for some $\gamma \in\mathbb{R}$.
\end{thm}

Finally, we remark that the essence of Theorems \ref{T2.4}.~and \ref{T2.5}.~is that,
in case $\alpha\notin\left\{0,1\right\}$, the algebraic properties listed in
Theorems \ref{T2.4}.~and \ref{T2.5}., respectively, determine the information measure
$(D_{n}^{\alpha})$ up to a multiplicative constant \emph{without any regularity assumption}.
Moreover, if $\alpha\in\left\{0,1\right\}$, then the mentioned algebraic properties
with a really mild regularity condition determine $(D_{n}^{\alpha})$ up to a multiplicative constant.

\end{document}